\documentclass{amsart}
\usepackage{graphicx,amssymb,enumerate}
\usepackage{hyperref}

\def\MR#1{\href{http://www.ams.org/mathscinet-getitem?mr=#1}{MR#1}}

\newcommand{\E}{\operatorname{E}}

\newtheorem{theorem}{Theorem}[section]
\newtheorem{lemma}{Lemma}[section]
\newtheorem{corollary}{Corollary}[section]

\theoremstyle{definition}
\newtheorem{algorithm}{Algorithm}[section]
\newtheorem{remark}{Remark}[section]

\title[Fast computation of the potential of charges on a line]{A fast simple
algorithm for computing the potential of charges on a line}

\author[Gimbutas]{Zydrunas Gimbutas}
\address{National Institute of Standards and Technology, Boulder, CO  80305,
USA}
\email{zydrunas.gimbutas@nist.gov}

\author[Marshall]{Nicholas F. Marshall}
\address{Department of Mathematics, Princeton University, Princeton, NJ 08540,
USA}
\email{nicholas.marshall@princeton.edu}
\thanks{N.F.M. was supported in part by NSF DMS-1903015}

\author[Rokhlin]{Vladimir Rokhlin}
\address{Program in Applied Mathematics, Yale University, New Haven, CT 06511,
USA} 
\email{vladimir.rokhlin@yale.edu}
\thanks{V.R. was supported in part by AFOSR FA9550-16-1-0175 and ONR
N00014-14-1-0797.}


\begin{document}

\begin{abstract}
We present a fast method for evaluating expressions of the form
$$
u_j = \sum_{i = 1,i \not = j}^n \frac{\alpha_i}{x_i - x_j}, \quad \text{for} \quad 
j = 1,\ldots,n,
$$
where $\alpha_i$ are real numbers, and $x_i$ are points in a compact interval of
$\mathbb{R}$. This expression can be viewed as representing the electrostatic
potential generated by charges on a line in $\mathbb{R}^3$. While fast
algorithms for computing the electrostatic potential of general distributions of
charges in $\mathbb{R}^3$ exist, in a number of situations in computational
physics it is useful to have a simple and extremely fast method for evaluating
the potential of charges on a line; we present such a method in this paper, and
report numerical results for several examples.  
\end{abstract}

\maketitle

\section{Introduction and motivation}
\subsection{Introduction}
In this paper, we describe a simple fast algorithm for evaluating expressions of
the form 
\begin{equation} \label{eq1}
u_j = \sum_{i=1, i \not = j}^n \frac{\alpha_i}{x_i - x_j}, \quad \text{for}
\quad j = 1,\ldots,n,
\end{equation} 
where $\alpha_i$ are real numbers, and $x_i$ are points in a
compact interval of $\mathbb{R}$. This expression can be viewed as representing
the electrostatic potential generated by charges on a line in $\mathbb{R}^3$.
We remark that fast algorithms for computing the electrostatic potential
generated by general distributions of charges in $\mathbb{R}^3$ exist, see for
example the Fast Multipole Method \cite{MR936632} whose relation to the method
presented in this paper is discussed in \S \ref{relatedworks}. However, in a number of
situations in computational physics it is useful to have a simple and extremely
fast method for evaluating the potential of charges on a line; we present such a
method in this paper.  Under mild assumptions the presented method  involves
$\mathcal{O}(n \log n)$ operations and has a small constant.  The method is
based on writing the potential $1/r$ as
$$
\frac{1}{r} = \int_0^\infty e^{-r t} dt.
$$
We show that there exists a small set of quadrature nodes $t_1,\ldots,t_m$
and weights $w_1,\ldots,w_m$ such that for a large range of values of $r$ we
have
\begin{equation} \label{oapprox}
\frac{1}{r} \approx \sum_{j=1}^m w_j e^{-r t_j},
\end{equation}
see Lemma \ref{quadlem}, which is a quantitative version of \eqref{oapprox}.
Numerically the nodes $t_1,\ldots,t_m$ and weights $w_1,\ldots,w_m$ are
computed using a procedure for constructing generalized Gaussian quadratures,
see \S \ref{nodesandweights}.  An advantage of representing $1/r$ as a sum
of exponentials is that the translation operator 
\begin{equation} \label{transeq}
\frac{1}{r} \mapsto \frac{1}{r+r'}
\end{equation}
can be computed by taking an inner product of the weights
$(w_1,\ldots,w_m)$ with a diagonal transformation of the vector
$(e^{-r t_1},\ldots,e^{-r t_m})$. Indeed, we have
\begin{equation} \label{diageq}
\frac{1}{r+r'} \approx  \sum_{j=1}^m
w_j e^{-(r +r') t_j} = \sum_{j=1}^m
w_j e^{-r' t_j} e^{-r t_j}.
\end{equation}
The algorithm described in \S \ref{algomain} leverages the existence of this
diagonal translation operator to efficiently evaluate \eqref{eq1}. 

\subsection{Relation to past work} \label{relatedworks}
We emphasize that fast algorithms for computing the potential generated by
arbitrary distributions of charges in $\mathbb{R}^3$ exist. An example of such
an algorithm is the Fast Multipole Method that was introduced by \cite{MR936632}
and has been extended by several authors including \cite{MR2558773, MR1489257, MR1273161}. In this paper, we present a simple scheme for the special case
where the charges are on a line, which occurs in a number of numerical
calcuations, see \ref{motivation}. The presented scheme has a much smaller
runtime constant compared to general methods, and is based on the diagonal form
\eqref{diageq} of the translation operator \eqref{transeq}. The idea of using
the diagonal form of this translation operator to accelerate numerical
computations has been studied by several authors; in particular, the diagonal
form is used in algorithms by Dutt, Gu and Rokhlin \cite{MR1411845}, and Yavin
and Rokhlin \cite{MR1675269} and was subsequently studied in detail by Beylkin
and Monz\'on \cite{MR2595881,MR2147060}.  

The current paper improves upon these past works by taking advantage of robust
generalized Gaussian quadrature codes \cite{MR2671296} that were not previously
available; these codes construct a quadrature rule that is exact for functions
in the linear span of a given Chebyshev system, and can be viewed as
a constructive version of Lemma \ref{krein} of Kre\u{\i}n \cite{MR0113106}. The
resulting fast algorithm presented in \S \ref{algomain} simplifies past
approaches, and has a small runtime constant; in particular, its computational
cost is similar to the computational cost of $5$-$10$ Fast Fourier Transforms on
data of a similar length, see \ref{numerics}.

\subsection{Motivation} \label{motivation}
Expressions of the form \eqref{eq1} appear in a number of situations in
computational physics. In particular, such expressions arise in connection with
the Hilbert Transform
$$
H f(x) = \lim_{\varepsilon \rightarrow 0} \frac{1}{\pi} \int_{|x-y|\ge
\varepsilon} \frac{f(y)}{y - x} dy.
$$
For example, the computation of the projection $P_m f$ of a function $f$
onto the first $m+1$ functions in a family of orthogonal polynomials can be
reduced to an expression of the form \eqref{eq1} by using the
Christoffel--Darboux formula, which is related to the
Hilbert transform; we detail the reduction of $P_m f$ to an expression of the
form \eqref{eq1} in the following.

Let $\{p_k\}_{k=0}^\infty$ be a family of monic polynomials that are orthogonal
with respect to the weight $w(x) \ge 0$ on $(a,b) \subseteq \mathbb{R}$.
Consider the projection operator 
$$
P_m f(x) := \int_a^b \sum_{k=0}^m \frac{p_k(x) p_k(y)}{h_k} f(y) w(y) dy,
$$
where $h_k := \int_a^b p_k(x)^2 w(x) dx$. Let $x_1,\ldots,x_n$ and
$w_1,\ldots,w_n$ be the $n > m/2$ point Gaussian quadrature nodes and
weights associated with $\{p_k\}_{k=0}^\infty$, and set
\begin{equation} \label{eq3}
u_j := \sum_{i=1}^n \sum_{k=0}^m 
\frac{p_k(x_j) p_k(x_i)}{h_k} f(x_i) w(x_i) ,
\quad \text{for} \quad j = 1,\ldots,n.
\end{equation}
By construction
the polynomial that interpolates the values $u_1,\ldots,u_n$ at the points
$x_1,\ldots,x_n$ will accurately approximate $P_m f$ on $(a,b)$ when
$f$ is sufficiently smooth, see for example \S 7.4.6 of Dahlquist and
Bj\"orck \cite{DahlquistBjorck1974}.  Directly evaluating \eqref{eq3} would
require $\Omega(n^2)$ operations. In contrast, the algorithm of this paper
together with the Christoffel--Darboux Formula can be used to evaluate 
\eqref{eq3} in $\mathcal{O}(n \log n)$ operations.  The
Christoffel-Darboux formula states that
\begin{equation} \label{eq4}
\sum_{k=0}^m \frac{p_k(x) p_k(y)}{h_k} = \frac{1}{h_m} \frac{p_{m+1}(x) p_m(y)
-p_m(x) p_{m+1}(y)}{x-y},
\end{equation}
see \S 18.2(v) of \cite{nist}. Using \eqref{eq4} to rewrite \eqref{eq3} yields
\begin{equation} \label{eq5}
u_j = \frac{1}{h_m}\left( f(x_j) +  \sum_{i=1, i \not = j}^m \frac{p_{m+1}(x_j)
p_m(x_i) -p_m(x_j) p_{m+1}(x_i)}{x_j-x_i} f(x_i) w(x_i)  \right),
\end{equation}
where we have used the fact that the diagonal term of the double summation is
equal to $f(x_j)/h_m$. The summation in \eqref{eq5} can be rearranged into two
expressions of the form \eqref{eq1}, and thus the method of this paper can be
used to compute a representation of $P_m f$ in $\mathcal{O}(n \log n )$
operations.

\begin{remark}
Analogs of the Christoffel--Darboux formula hold for many other families of
functions; for example, if $J_{\nu}(w)$ is a Bessel function of the first kind,
then  we have
$$
\sum_{k=1}^\infty 2(\nu +k) J_{\nu +k}(w) J_{v+k}(z) = \frac{w z}{w - z} \left(
J_{\nu+1}(w) J_\nu(z) - J_\nu(w) J_{\nu+1}(z) \right),
$$
see \cite{Tygert2006}. This formula can be used to write a projection operator
related to Bessel functions in an analogous form to \eqref{eq5}, and the
algorithm of this paper can be similarly applied
\end{remark}

\begin{remark}
A simple modification of the algorithm presented in this paper can be used to
evaluate more general expressions of the form
$$
v_j = \sum_{i=1}^n \frac{\alpha_i}{x_i - y_j}, \quad \text{for} \quad j =
1,\ldots,m,
$$
where $x_1,\ldots,x_n$ are source points, and $y_1,\ldots,y_m$ are target
points. For simplicity, this paper focuses on the case where the source and
target points are the same, which is the case in the projection application
described above.
\end{remark}

\section{Main result}

\subsection{Main result} \label{mainresult}
Our principle analytical result is the following theorem, which provides precise
accuracy and computational complexity guarantees for the algorithm presented in
this paper, which is detailed in \S \ref{algomain}.

\begin{theorem} \label{thm1}
Let $x_1 <\ldots <x_n \in [a,b]$ and $\alpha_1,\ldots,\alpha_n
\in \mathbb{R}$ be given. Set
$$
u_j := \sum_{i=1,i \not = j}^n \frac{\alpha_i}{x_i - x_j}, \quad \text{for}
\quad j = 1,\ldots,n.  
$$
Given $\delta >0$ and $\varepsilon > 0$, the algorithm
described in \S \ref{algomain} computes values $\tilde{u}_j$  such
\begin{equation} \label{errest}
\frac{\left| \tilde{u}_j - u_j \right|}{\sum_{i=1}^n |\alpha_i|} \le \varepsilon
, \quad \text{for} \quad j = 1,\ldots,n
\end{equation}
in $\mathcal{O} \left(n \log (\delta^{-1}) \log(\varepsilon^{-1}) +  N_\delta 
\right)$ operations, where
\begin{equation} \label{ndelta}
N_\delta := \sum_{j=1}^n \# \{ x_i : |x_j - x_i| < \delta (b-a) \}.
\end{equation}

\end{theorem}

The proof of Theorem \ref{thm1} is given in \S \ref{proofmainresult}.  Under
typical conditions,  the presented algorithm involves $\mathcal{O}( n \log n)$
operations. The following corollary describes a case of interest, where the
points $x_1,\ldots,x_n$ are Chebyshev nodes for a compact interval $[a,b]$ (we
define Chebyshev nodes in \S \ref{preliminaries}).

\begin{corollary} \label{cor1}
Fix $\varepsilon = 10^{-15}$, and let the points $x_1,\ldots,x_n$ be
Chebyshev nodes on $[a,b]$. If $\delta = 1/n$, then the algorithm of \S
\ref{algomain} involves $\mathcal{O}( n \log n)$ operations.
\end{corollary}

The proof of Corollary \ref{cor1} is given in \S \ref{completeproof}.  The
following corollary states that a similar result holds for uniformly random
points.

\begin{corollary} \label{cor2}
Fix $\varepsilon = 10^{-15}$, and suppose that $x_1,\ldots,x_n$ are
sampled uniformly at random from $[a,b]$. If $\delta = 1/n$, then the
algorithm of \S \ref{algomain} involves $\mathcal{O}( n \log n)$ operations
with high probability.
\end{corollary}

The proof of Corollary \ref{cor2} is immediate from standard probabilistic
estimates.  The following remark describes an adversarial configuration of
points.

\begin{remark} \label{rmk1}
Fix $\varepsilon > 0$, and let $x_1,\ldots,x_{2 n}$ be a collection of
points such that $x_1,\ldots,x_n$ and $x_{n+1},\ldots,x_{2 n}$ are evenly spaced
in $[0,2^{-n}]$ and $[1-2^{-n},1]$, respectively, that is
$$
x_j = 2^{-n} \left( \frac{j-1}{n -1} \right),
\quad \text{and} \quad
x_{n+j} = 1 +2^{-n} \left( \frac{j-n}{n -1} \right),
\quad \text{for} \quad j = 1,\ldots,n.
$$
We claim that Theorem \ref{thm1} cannot guarantee a complexity better than
$\mathcal{O}(n^2)$ for this configuration of points. Indeed, if $\delta \ge
2^{-n}$, then $N_\delta \ge n^2/2$, and if $\delta < 2^{-n}$, then
$\log_2(\delta^{-1}) > n$.  In either case
$$
n \log(\delta^{-1}) + N_\delta = \Omega(n^2).
$$
This complexity is indicative  of the performance of the algorithm for this
point configuration; the reason that the algorithm performs poorly is that
structures exist at two different scales.  If such a configuration were
encountered in practice, it would be possible to modify the algorithm of \S
\ref{algomain} to also involve two different scales to achieve evaluation in
$\mathcal{O}(n \log n)$ operations.
\end{remark}

\section{Algorithm} \label{algomain}
\subsection{High level summary} \label{high}
The algorithm involves passing over the points $x_1,\ldots,x_n$ twice. First, we
pass over the points in ascending order and compute
\begin{equation} \label{uplus}
\tilde{u}_j^+ \approx \sum_{i=1}^{j-1} \frac{\alpha_i}{x_i - x_j},
\quad \text{for} \quad j = 1,\ldots,n,
\end{equation}
and second, we pass over the points in descending order and compute
\begin{equation} \label{uminus}
\tilde{u}_j^- \approx \sum_{i=j+1}^{n} \frac{\alpha_i}{x_i - x_j},
\quad \text{for} \quad j = 1,\ldots,n.
\end{equation}
Finally, we define  $ \tilde{u}_j := \tilde{u}_j^+ + \tilde{u}_j^-$ for $j =
1,\ldots,n$ such that
$$
\tilde{u}_j \approx \sum_{i=1, i \not = j}^n \frac{\alpha_i}{x_i  - x_j}, \quad
\text{for} \quad j = 1,\ldots,n.
$$
We call the computation of $\tilde{u}_1^+,\ldots,\tilde{u}_n^+$ the forward pass
of the algorithm, and the computation of $\tilde{u}_1^-,\ldots,\tilde{u}_n^+$
the backward pass of the algorithm. The forward pass of the algorithm computes
the potential generated by all points to the left of a given point, while the
backward pass of the algorithm computes the potential generated by all points to
the right of a given point. In \S \ref{informal} and \S \ref{detailed} we give
an informal and detailed description of the forward pass of the algorithm.  The
backward pass of the algorithm is identical except it considers the points in
reverse order.

\subsection{Informal description}  \label{informal}
In the following, we give an informal description of the forward pass of the
algorithm that computes
$$
\tilde{u}_j^+ \approx \sum_{i=1}^{j-1} \frac{\alpha_i}{x_i - x_j},
\quad \text{for} \quad j = 1,\ldots,n.
$$
Assume that a small set of nodes
$t_1,\ldots,t_m$ and weights $w_1,\ldots,w_m$ such that
\begin{equation} \label{comp2}
\frac{1}{r} \approx \sum_{i=1}^m w_i e^{-r t_i} \quad
\text{for} \quad r \in [\delta(b-a),b-a],
\end{equation}
where $\delta > 0$ is given and fixed. The existence and computation of
such nodes and weights is described in \S \ref{completeproof} and \S
\ref{nodesandweights}. We divide the sum defining $u_j^+$ into two parts:
\begin{equation} \label{divide}
\tilde{u}_j^+ \approx \sum_{i=1}^{j_0} \frac{\alpha_i}{x_i - x_j}
+ \sum_{i=j_0+1}^{j-1} \frac{\alpha_i}{x_i - x_j},
\end{equation}
where $j_0 = \max \big\{ i  : x_j - x_i > \delta (b-a)
\big\}.$ By definition, the points $x_1,\ldots,x_{j_0}$ are all distance at
least $\delta(b-a)$ from $x_j$. Therefore, by \eqref{comp2} 
$$
\tilde{u}_j^+ \approx - \sum_{i=1}^{j_0} \sum_{k=1}^m  w_k \alpha_i e^{-(x_j-x_i)
t_k} + \sum_{i=j_0+1}^{j-1} \frac{\alpha_i}{x_i - x_j}.
$$
If we define
\begin{equation} \label{geq}
g_k(j_0) = \sum_{i=1}^{j_0} \alpha_i e^{-(x_{j_0} - x_i)t_k}, \quad
\text{for} \quad k = 1,\ldots,m,
\end{equation}
then it is straightforward to verify that
\begin{equation} \label{maineqap}
\tilde{u}_j^+ \approx - \sum_{k=1}^m w_k g_k(j_0) e^{-(x_j - x_{j_0}) t_k} +
\sum_{i=j_0+1}^{j-1} \frac{\alpha_i}{x_i - x_j}.
\end{equation}
Observe that we can update $g_k(j_0)$ to $g_k(j_0+1)$ using the following
formula 
\begin{equation} \label{updateg}
g_k(j_0+1) =  \alpha_{j_0} + e^{-(x_{j_0+1} - x_{j_0}) t_k} g_k(j_0), \quad
\text{for} \quad k = 1,\ldots,m.
\end{equation}
We can now summarize the algorithm for computing
$\tilde{u}_1^+,\ldots,\tilde{u}_n^+$. For each $j$, we compute
$\tilde{u}_j^+$ by the following three steps:
\begin{enumerate}[\quad 1.]
\item Update $g_1,\ldots,g_m$ as necessary
\item Use $g_1,\ldots,g_m$ to evaluate the potential from $x_i$ such
that
$x_j - x_i > \delta (b-a)$
\item Directly evaluate the potential from $x_i$ such that $0 < x_j - x_i < \delta (b-a)$
\end{enumerate}
By \eqref{updateg}, each update of $g_1,\ldots,g_m$ requires $\mathcal{O}(m)$
operations, and we must update $g_1,\ldots,g_m$ at most $n$ times, so we
conclude that the total cost of the first step of the algorithm is
$\mathcal{O}(n m)$ operations. For each $j = 1,\ldots,n$, the second and third
step of the algorithm involve $\mathcal{O}(m)$ and $\mathcal{O}(\# \{x _i : 0 <
x_j - x_i < \delta(b-a)\})$ operations, respectively, see \eqref{maineqap}. It
follows that the total cost of the second and third step of the algorithm is
$\mathcal{O}(n m + N_\delta)$ operations, where $N_\delta$ is defined in
\eqref{ndelta}. We conclude that $\tilde{u}_1^+,\ldots,\tilde{u}_n^+$ can be
computed in $\mathcal{O}( n m + N_\delta)$ operations.  In \S
\ref{proofmainresult}, we complete the proof of the computational complexity
guarantees of Theorem \ref{thm1} by showing that there exist $m = \mathcal{O}(
\log(\delta^{-1}) \log( \varepsilon^{-1}) )$ nodes $t_1,\ldots,t_m$ and weights
$w_1,\ldots,w_m$ that satisfy \eqref{comp2}, where $\varepsilon > 0$ is the
approximation error in \eqref{comp2}.

\subsection{Detailed description} \label{detailed} 
\label{algorithm} In the following, we give a detailed description of the
forward pass of the algorithm that computes
$\tilde{u}_1^+,\ldots,\tilde{u}_n^+$. Suppose that $\delta > 0$ and $\varepsilon
> 0$ are given and fixed.  We describe the algorithm under the assumption that
we are given quadrature nodes $t_1,\ldots,t_m$ and weights $w_1,\ldots,w_m$ such
that
\begin{equation} \label{quad}
\left| \frac{1}{r} - \sum_{j=1}^m  w_j e^{-r t_j} \right| \le \varepsilon
\quad \text{for} \quad 
r \in [\delta (b-a), b-a].
\end{equation}
The existence of such weights and nodes is established in \S
\ref{completeproof}, and the computation of such nodes and weights is discussed
in \S \ref{nodesandweights}.  To simplify the description of the algorithm, we
assume that $x_0 = -\infty$ is a placeholder node that does not generate a
potential.

\begin{algorithm} \label{algo1}
\textit{Input:} $x_1 < \cdots < x_n \in [a,b]$,
$\alpha_1,\ldots,\alpha_n \in \mathbb{R}$.
\textit{Output:} $\tilde{u}_1^+,\ldots,\tilde{u}_n^+$.

\begin{enumerate}[\quad 1:]
\item \qquad $j_0 = 0$ and $g_1 = \cdots = g_m = 0$
\item
\item \qquad \textit{main loop:}
\item \qquad \textbf{for} $j = 1,\ldots,n$
\item
\item \qquad \qquad \textit{update $g_1,\ldots,g_m$ and $j_0$:}
\item \qquad \qquad \textbf{while} $x_j - x_{j_0+1} > \delta(b-a)$ 
\item \qquad \qquad \qquad \textbf{for} i = 1,\ldots,m
\item \qquad \qquad \qquad \qquad
$g_i = g_i e^{-(x_{j_0+1} - x_{j_0}) t_i}+ \alpha_i$
\item \qquad \qquad \qquad \textbf{end for}
\item \qquad \qquad \qquad $j_0 = j_0 + 1$
\item \qquad \qquad \textbf{end while}
\item
\item \qquad \qquad \textit{compute potential from $x_i$ such that $x_i \le x_{j_0}:$}
\item \qquad \qquad $\tilde{u}_j^+ = 0$
\item \qquad \qquad \textbf{for} $i = 1,\ldots,m$
\item \qquad \qquad \qquad 
$\tilde{u}_j^+ = \tilde{u}_j^+ - w_i g_i e^{-(x_j - x_{j_0}) t_i}$
\item \qquad \qquad \textbf{end for}
\item
\item \qquad \qquad \textit{compute potential from $x_i$ such that $x_{j_0+1} \le x_i \le x_{j-1}$}
\item \qquad \qquad \textbf{for} $i = j_0+1,\ldots,j-1$
\item \qquad \qquad \qquad 
$\tilde{u}_j^+ = \tilde{u}_j^+ +  \alpha_i/(x_i - x_j).$
\item \qquad \qquad \textbf{end for}
\item \qquad \textbf{end for}
\end{enumerate}
\end{algorithm}

\begin{remark}  \label{precomputation}
In some applications, it may be necessary to evaluate an expression of the form
\eqref{eq1} for many different weights $\alpha_1,\ldots,\alpha_n$ associated
with a fixed set of points $x_1,\ldots,x_n$.  For example, in the projection
application described in \S \ref{motivation} the weights
$\alpha_1,\ldots,\alpha_n$ correspond to a function that is being projected,
while the points $x_1,\ldots,x_n$ are a fixed set of quadrature nodes.  In such
situations, pre-computing the exponentials $e^{-(x_j - x_{j_0}) t_i}$ used in
the Algorithm \ref{algo1} will significantly improve the runtime, see
\S \ref{numres}.
\end{remark}

\section{Proof of Main Result} \label{proofmainresult}
\subsection{Organization}
In this section we complete the proof of Theorem \ref{thm1}; the section is
organized as follows. In \S \ref{preliminaries} we give mathematical
preliminaries. In \S \ref{tech} we state and prove two technical lemmas.
In \S \ref{completeproof} we prove Lemma \ref{quadlem}, which together with
the analysis in \S \ref{algomain} establishes Theorem \ref{thm1}. In
\S \ref{proofcor} we prove  Corollary \ref{cor1}, and Corollary \ref{cor2}.

\subsection{Preliminaries} \label{preliminaries}
Let $a < b \in \mathbb{R}$ and $n \in \mathbb{Z}_{> 0}$ be fixed, and
suppose that $f : [a,b] \rightarrow \mathbb{R}$, and $x_1 < \cdots < x_n \in
[a,b]$ are given. The interpolating polynomial $P$ of the function $f$ at
$x_1,\ldots,x_n$ is the unique polynomial of degree at most $n-1$ such that $$
P(x_j) = f(x_j), \quad \text{for} \quad j = 1,\ldots,n.
$$
This interpolating polynomial $P$ can be explicitly defined by
\begin{equation} \label{P}
P(x) = \sum_{j=1}^n f(x_j) q_{j}(x),
\end{equation}
where $q_j$ is the nodal polynomial for $x_j$, that is,
\begin{equation} \label{nodal}
q_{j}(x) = \prod_{k = 1,k \not = j}^n \frac{x - x_k}{x_j - x_k}.
\end{equation}
We say $x_1,\ldots,x_n$ are Chebyshev nodes for the interval $[a,b]$ if
\begin{equation} \label{nodes}
x_j = \frac{b+a}{2} + \frac{b-a}{2} \cos \left( \pi  \frac{j - \frac{1}{2}}{ n
} \right), \quad \text{for} \quad j = 1,\ldots,n.
\end{equation}
The following lemma is a classical result in approximation theory. It says that
a smooth function on a compact interval is accurately approximated by the
interpolating polynomial of the function at Chebyshev nodes, see for example \S
4.5.2 of Dahlquist and Bj\"orck \cite{DahlquistBjorck1974}.

\begin{lemma} \label{lem1}
Let $f \in C^{n}([a,b])$, and $x_1,\ldots,x_n$ be Chebyshev nodes for $[a,b]$.
If $P$ is the interpolating polynomial for $f$ at $x_1,\ldots,x_n$, then
$$
\sup_{x \in [a,b]} |f(x) - P(x)| \le \frac{2 M}{n !} \left( \frac{b-a}{4}
\right)^{n},
$$
where 
$$
M = \sup_{x \in [a,b]} |f^{(n)}(x)|.
$$
\end{lemma}

In addition to Lemma \ref{lem1}, we require a result about the existence of
generalized Gaussian quadratures for Chebyshev systems. In 1866, Gauss
\cite{Gauss1866} established the existence of quadrature nodes $x_1,\ldots,x_n$
and weights $w_1,\ldots,w_n$ for an interval $[a,b]$ such that
$$
\int_a^b f(x) dx = \sum_{j=1}^n w_j f(x_j),
$$
whenever $f(x)$ is a polynomial of degree at most $2n - 1$. This result was
generalized from polynomials to Chebyshev systems by Kre\u{\i}n
\cite{MR0113106}. A collection of functions $f_0,\ldots,f_n$ on $[a,b]$ is a
Chebyshev system if every nonzero generalized polynomial 
$$
g(t) = a_0 f_0(t) + \cdots + a_n f_n(t), \quad \text{for} \quad a_0,\ldots,a_n
\in \mathbb{R},
$$
 has at most $n$ distinct zeros in $[a,b]$.  The following result of Kre\u{\i}n
says that any function in the span of a Chebyshev system of $2n$ functions can
be integrated exactly by a quadrature with $n$ nodes and $n$ weights.

\begin{lemma}[Kre\u{\i}n \cite{MR0113106}] \label{krein}
Let $f_0,\ldots,f_{2n-1}$ be a Chebyshev system of continuous functions on
$[a,b]$, and $w : (a,b) \rightarrow \mathbb{R}$ be a continuous positive weight
function. Then, there exists unique nodes $x_1,\ldots,x_n$ and weights
$w_1,\ldots,w_n$ such that 
$$
\int_a^b f(x) w(x) dx = \sum_{j=1}^n w_j f(x_j),
$$
whenever $f$ is in the span of $f_0,\ldots,f_{2n-1}$.
\end{lemma}

\subsection{Technical Lemmas} \label{tech}
In this section, we state and prove two technical lemmas that are
involved in the proof of Theorem \ref{thm1}. We remark that a similar version of
Lemma \ref{lem2} appears in \cite{Rokhlin1988}.

\begin{lemma} \label{lem2}
Fix $a > 0$ and $t \in [0,\infty)$, and let $r_1,\ldots,r_n$ be Chebyshev nodes
for $[a,2 a]$.  If $P_{t}(r)$ is the interpolating polynomial for $e^{-r t}$
at $r_1,\ldots,r_n$, then
$$
\sup_{r \in [a,2 a]} \left| e^{-r t} - P_{t}(r) \right|\le \frac{1}{4^n}.
$$
\end{lemma}

\begin{proof}
We have
$$
\sup_{r \in [a,2 a]} \left| \frac{\partial^n}{\partial r^n} e^{-r t} \right| =
\sup_{r \in [a,2 a]} |t^n e^{-r t}| = t^n e^{-t a}.
$$
By writing the derivative of $t^n e^{-t a}$ as 
$$
\frac{d}{d t} t^n e^{-t a} = \left( \frac{n}{a} -t \right) a
t^{n-1} e^{-a t},
$$
we can deduce that the maximum of $t^n e^{-t a}$ occurs at $t = n/a$, that is,
\begin{equation} \label{maxna}
\sup_{t \in [0,\infty)} t^n e^{-t a} = \left( \frac{n}{a} \right)^n e^{-a(n/a)}.
\end{equation}
By \eqref{maxna} and the result of Lemma \ref{lem1}, we conclude that
$$
\sup_{t \in [a,2a]} |e^{-r t} - P_t(r)| \le \frac{2 (n/a)^n e^{-a(n/a)}}{n !} \left(
\frac{a}{4} \right)^{n} = \frac{2 n^n e^{-n}}{n!} \frac{1}{4^n}.
$$
It remains to show that $2 n^n e^{-n} \le n!$. Since $\ln(x)$ is a increasing
function, we have
$$
n \ln n - n  + 1 = \int_1^n \ln(x) dx \le \int_1^n \sum_{j=1}^{n-1}
\chi_{[j,j+1]}(x) \ln(j+1) dx = \sum_{j=1}^n \ln(j).
$$
Exponentiating both sides of this inequality gives $e n^n e^{-n} \le n!$, which
is a classical inequality related to Stirling's approximation. This completes
the proof.
\end{proof}

\begin{lemma} \label{approx}
Suppose that $\varepsilon > 0$ and $M > 1$ are given. Then, there exists 
$$
m = \mathcal{O}(\log(M) \log(\varepsilon^{-1}))
$$ 
values $r_1,\ldots,r_m \in [1,M]$ such that for all $r \in [1,M]$ we have
\begin{equation} \label{approxeq}
\sup_{t \in [0,\infty)} \left| e^{-r t} - \sum_{j=1}^m c_j(r)
e^{-r_j t} \right| \le \varepsilon,
\end{equation}
for some choice of coefficients $c_j(r)$ that depend on $r$.
\end{lemma}

\begin{proof}
We construct an explicit set of $m := (\lfloor \log_2 M \rfloor +1) ( \lfloor
\log_4 \varepsilon^{-1}  \rfloor + 1)$ points and coefficients such that
\eqref{approxeq} holds. Set $n := \lfloor \log_4 \varepsilon^{-1} \rfloor +1$.
We define the points $r_1,\ldots,r_m$ by
\begin{equation} \label{nodefin}
r_{i n +k} := 2^{i-1} \left( 3 + \cos \left( \pi \frac{k - \frac{1}{2}}{n}
\right) \right), 
\end{equation}
for $k = 1,\ldots,n$ and $i = 0,\ldots,\lfloor \log_2 M \rfloor$, and define
the coefficients $c_1(r),\ldots,c_m(r)$ by
\begin{equation} \label{coeff}
c_{i n +k}(r) := \chi_{[2^{i},2^{i+1})}(r)  \prod_{l=1,l \not = k}^{\lfloor
\log_4 \varepsilon^{-1} \rfloor} \frac{r - r_{i n+l}}{r_{i n + l} -r_{i n
+ k}},
\end{equation}
for $k = 1,\ldots,n$ and $i = 0,\ldots,\lfloor \log_2 M \rfloor$. We claim that
$$
\sup_{r \in [1,M]} \sup_{t \in [0,\infty)} \left| e^{-r t} - \sum_{j=1}^m c_j(r)
e^{-r_j t} \right| \le \varepsilon.
$$ 
Indeed, fix $r \in [1,M]$, and let $i_0 \in \{0,\ldots,\lfloor \log_2 M
\rfloor\}$ be the unique integer such that $r \in [2^{i_0},2^{i_0+1})$. By
the definition of the coefficients, see \eqref{coeff}, we have
$$
\sum_{j=1}^m c_j(r) e^{-r_j t} = \sum_{k=1}^n e^{-r_{i_0 n +k} t} \prod_{l=1,l
\not = k}^{\lfloor \log_4 \varepsilon^{-1} \rfloor} \frac{r - r_{i_0
n+l}}{r_{i_0 n + l} -r_{i_0 n + k}}.
$$
We claim that the right hand side of this equation  is the interpolating
polynomial $P_{t,i_0}(r)$ for $e^{-r t}$ at $r_{i_0 n + k},\ldots,r_{(i_0+1)n}$,
that is,
$$
\sum_{k=1}^n e^{-r_{i_0 n +k} t} \prod_{l=1,l \not = k}^{\lfloor \log_4
\varepsilon^{-1} \rfloor} \frac{r - r_{i_0 n+l}}{r_{i_0 n + l} -r_{i_0 n + k}} =
P_{t,i_0}(r).
$$
Indeed, see \eqref{P} and \eqref{nodal}. Since the points $r_{i_0 n +
k},\ldots,r_{(i_0+1)n}$ are Chebyshev nodes for the interval
$[2^{i_0},2^{i_0+1}]$, and since $i_0$ was chosen such that $r \in
[2^{i_0},2^{i_0+1})$, it follows from Lemma \ref{lem2} that
$$
\left| e^{-r t} - P_{t,i_0}(r) \right|\le \frac{1}{4^n}
\quad \text{for} \quad t \in [0,\infty).
$$
Since $n = \lfloor \log_4 \varepsilon^{-1} \rfloor
+1$ the proof is complete.
\end{proof}

\begin{remark}
The proof of Lemma \ref{approx} has the additional consequence that the
coefficients $c_1(r),\ldots,c_m(r)$ in \eqref{approxeq} can be chosen such that
they satisfy
$$
|c_j(r)| \le \sqrt{2} \quad \text{for} \quad j=1,\ldots,m.
$$
Indeed, in \eqref{coeff} the coefficients $c_j(r)$  are either equal zero or
equal to the nodal polynomial, see \eqref{nodal}, for Chebyshev nodes on an
interval that contains $r$. The nodal polynomials for Chebyshev nodes on an
interval $[a,b]$ are bounded by $\sqrt{2}$ on $[a,b]$, see for example
\cite{Rokhlin1988}.  The fact that $e^{-r t}$ can be approximated as a linear
combination of functions $e^{-r_1 t},\ldots,e^{-r_m t}$
with small coefficients means that the approximation of Lemma \ref{approx} can
be used in finite precision environments without any unexpected catastrophic
cancellation.
\end{remark}

\subsection{Completing the proof of Theorem \ref{thm1}} \label{completeproof}
Previously in \S \ref{informal}, we proved that the algorithm of \S
\ref{algomain} involves $\mathcal{O}( n m + N_\delta)$ operations. To complete
the proof of Theorem \ref{thm1} it remains to show that there exists
$$
m = \mathcal{O}( \log( \varepsilon^{-1} ) 
\log( \delta^{-1} ))
$$
points $t_1,\ldots,t_m$ and weights $w_1,\ldots,w_m$ that satisfy \eqref{quad};
we show the existence of such nodes and weights in the following lemma, and thus
complete the proof of Theorem \ref{thm1}. The computation of such nodes and
weights is described in \S \ref{nodesandweights}.

\begin{lemma} \label{quadlem}
Fix $a < b \in \mathbb{R}$, and let $\delta > 0$ and $\varepsilon > 0$ be given.
Then, there exists $m = \mathcal{O}( \log(\varepsilon^{-1}) \log(\delta^{-1}))$
nodes $t_1,\ldots,t_m$ and weights $w_1,\ldots,w_m$ such
that
\begin{equation} \label{star}
\left| \frac{1}{r} - \sum_{j=1}^m  w_j e^{-r t_j} \right| \le \varepsilon,
\quad 
\text{for} \quad r \in [\delta (b-a), b-a].
\end{equation}
\end{lemma}

\begin{proof} 
Fix $a < b \in \mathbb{R}$, and let $\delta, \varepsilon >0$  be given. By the
possibility of rescaling $r$, $w_j$, and $t_j$, we may assume that $b-a =
\delta^{-1}$ such that we want to establish \eqref{star} for $r \in
[1,\delta^{-1}]$.  By Lemma \ref{approx} we can choose $2 m = \mathcal{O} (
\log(\varepsilon^{-1}) \log(\delta^{-1}))$ points $r_0,\ldots,r_{2m-1} \in
[1,\delta^{-1}]$, and coefficients $c_0(r),\ldots,c_{2m-1}(r)$ depending on $r$
such that
\begin{equation} \label{erri}
\sup_{r \in [1,\delta^{-1}]} \sup_{t \in [0,\infty)} \left| e^{-r t} -
\sum_{j=0}^{2m - 1} c_j(r) e^{-r_j t} \right| \le \frac{\varepsilon}{2
\log(2\varepsilon^{-1})}.
\end{equation}
The collection of functions $e^{-r_0 t},\ldots,e^{-r_{2m-1} t}$ form a
Chebyshev system of continuous functions on the interval $[0,\log(2
\varepsilon^{-1})]$, see for example \cite{MR0204922}.  Thus, by Lemma
\ref{krein} there exists $m$ quadrature nodes $t_1,\ldots,t_m$ and weights
$w_1,\ldots,w_m$ such that
$$
\int_0^{\log(2 \varepsilon^{-1})} f(t) dt = \sum_{j=1}^m w_j f(t_j),
$$
whenever $f(t)$ is in the span of $e^{-r_0 t},\ldots,e^{-r_{2m-1} t}$. By the
triangle inequality
\begin{multline} \label{s1}
\left| \frac{1}{r} - \sum_{j=1}^m w_j e^{-r t_j} \right| \\ \le 
\left| \frac{1}{r} - \int_0^{ \log(2 \varepsilon^{-1})} e^{-r t} dt \right|
+ \left| \int_0^{\log(2 \varepsilon^{-1})} e^{-r t} dt - \sum_{j=1}^m w_j e^{r t_j}
\right|.
\end{multline}
Recall that we have assumed $r \in [1,\delta^{-1}]$, in particular, $r \ge 1$ so
it follows that
\begin{equation} \label{s2}
\left| \frac{1}{r} - \int_0^{ \log(2 \varepsilon^{-1})} e^{-r t} dt \right| \le
\varepsilon/2.
\end{equation}
By \eqref{erri}, the function $e^{-r t}$ can be approximated to error
$\varepsilon/(2 \log(2 \varepsilon^{-1}))$ in the $L^\infty$-norm on $[0,\log(2
\varepsilon^{-1})]$ by functions in the span of $e^{-r_0 t},\ldots,e^{-r_{2m-1}
t}$. Since our quadrature is exact for these functions, we conclude that
\begin{equation} \label{s3}
\left| \int_0^{\log(2 \varepsilon^{-1})} e^{-r t}dt - \sum_{j=1}^m w_j e^{r t_j}
\right| \le \varepsilon/2.
\end{equation}
Combining \eqref{s1}, \eqref{s2}, and \eqref{s3} completes the
proof.
\end{proof}

\subsection{Proof of Corollary \ref{cor1}} \label{proofcor}
In this section, we prove Corollary \ref{cor1}, which states that the algorithm
of \S \ref{algomain} involves $\mathcal{O}(n \log n)$ operations when
$x_1,\ldots,x_n$ are Chebyshev nodes, $\varepsilon = 10^{-15}$, and $\delta =
1/n$.

\begin{proof}[Proof of Corollary \ref{cor1}]
By rescaling the problem we may assume that $[a,b] = [-1,1]$
such that the Chebyshev nodes $x_1,\ldots,x_n$ are given by 
$$
x_j = \cos \left( \pi \frac{j - \frac{1}{2}}{n } \right), \quad \text{for}
\quad j = 1,\ldots,n.
$$
By the result of Theorem \ref{thm1}, it suffices to show that $N_\delta = \mathcal{O}(n
\log n)$, where 
$$
N_\delta := \sum_{j=1}^n \# \left\{ x_i : |x_j - x_i| < \frac{1}{n} \right\}.
$$
It is straightforward to verify that the number of Chebyshev nodes within an
interval of radius $1/n$ around the point $-1 < x < 1$ is
$\mathcal{O}(1/\sqrt{1-x^2})$, that is,
$$
\# \left\{ x_i : |x - x_i| < \frac{1}{n} \right\} = \mathcal{O} \left(
\frac{1}{\sqrt{1 - x^2}} \right), \quad \text{for} \quad -1 < x < 1.
$$
This estimate, together with the fact that the first and last Chebyshev node are
distance at least $1/n^2$ from $1$ and $-1$, respectively, gives the estimate
\begin{equation} \label{estc}
\sum_{j=1}^n \# \left\{ x_i : |x_j - x_i| < \frac{1}{n} \right\} = \mathcal{O}
\left( \int_{1/n^2}^{\pi-1/n^2} \frac{n}{\sqrt{1-\cos(t)^2}} dt \right).
\end{equation}
Let $\pi/2 > \eta > 0$ be a fixed parameter; direct calculation yields
$$
\int_{\eta}^{\pi-\eta} \frac{1}{\sqrt{1-\cos(t)^2}} dt  = 2 \log \left( \cot
\left( \frac{\eta}{2} \right) \right) = \mathcal{O} \left( \log \left(\eta^{-1}
\right)\right).
$$
Combining this estimate with \eqref{estc} yields $N_\delta = \mathcal{O}(n \log
n)$ as was to be shown.
\end{proof}

\section{Numerical results and implementation details} \label{numerics}

\subsection{Numerical results}  \label{numres}
We report numerical results for two different point distributions:
uniformly random points in $[1,10]$, and Chebyshev nodes in $[-1,1]$. 
In both cases, we choose the weights $\alpha_1,\ldots,\alpha_n$ uniformly at
random from $[0,1]$, and test the algorithm for 
$$
n = 1000 \times 2^k \quad \text{points}, \quad \text{for} \quad k = 0,\ldots,10.
$$
We time two different versions of the algorithm: a standard implementation,
and an implementation that uses precomputed exponentials.  Precomputing
exponentials may be advantageous in situations where the expression
\begin{equation} \label{eqnew}
u_j = \sum_{i=1}^n \frac{\alpha_i}{x_i - x_j}, \quad \text{for} \quad j =
1,\ldots,n,
\end{equation}
must be evaluated for many different weights $\alpha_1,\ldots,\alpha_n$
associated with a fixed set of points $x_1,\ldots,x_n$, see Remark
\ref{precomputation}. We find that using precomputed exponentials makes the
algorithm approximately ten times faster, see Tables \ref{key}, \ref{figrand},
and \ref{figcheb}. In addition to reporting timings, we report the absolute
relative difference between the output of the algorithm of \S \ref{algomain} and
the output of direct evaluation; we define the absolute relative difference
$\epsilon_r$ between the output $\tilde{u}_j$ of the algorithm of \S
\ref{algomain} and the output $u_j^d$ of direct calculation by
\begin{equation} \label{errdef}
\epsilon_r := \sup_{j = 1,\ldots,n} \left| \frac{\tilde{u}_j -
u^d_j}{\bar{u}_j} \right|,
\quad \text{where} \quad
\bar{u}_j := \sum_{i=1}^n \left| \frac{\alpha_i}{x_i - x_j} \right|,
\end{equation}
Dividing by $\bar{u}_j$ accounts were the fact that the calculations are
performed  in finite precision; any remaining loss of accuracy in the numerical
results is a consequence of the large number of addition and multiplication
operations that are performed. 
All calculations are performed in double precision, and the algorithm of \S
\ref{algomain} is run with $\varepsilon = 10^{-15}$. The parameter $\delta > 0$
is set via an empirically determined heuristic.  The numerical experiments were
performed on a laptop with a Intel Core i5-8350U CPU and $7.7$ GiB of memory;
the code was written in Fortran and compiled with gfortran with standard
optimization flags. The results are reported in Tables \ref{key}, \ref{figrand},
and \ref{figcheb}. 

To put the run time of the algorithm in context, we additionally perform a time
comparison to the Fast Fourier Transform (FFT), which also has complexity
$\mathcal{O}(n \log n)$. Specifically, we compare the run time of the algorithm
of \S \ref{algomain} on random data using precomputed exponentials with the run
time of an FFT implementation from FFTPACK \cite{fftpack} on random data of the
same length using precomputed exponentials.  We report these timings in Table
\ref{figfft}; we find that the FFT is roughly $5$-$10$ times faster than our
implementation of the algorithm of \S \ref{algomain}; we remark that no
significant effort was made to optimize our implementation, and that it may be
possible to improve the run time by vectorization. 

\begin{table}[h!]
\begin{tabular}{c|l}
Label & Definition \\
\hline
$n$ & number of points \\
$t_w$ & time of algorithm of \S \ref{algomain} without precomputation in seconds \\
$t_p$ & time of precomputing exponentials for algorithm of \S
\ref{algomain} in seconds \\ 
$t_u$ & time of algorithm of \S \ref{algomain} using precomputed exponentials in seconds \\
$t_d$ & time of direct evaluation in seconds \\
$\epsilon_r $ & maximum absolute relative difference defined in \eqref{errdef} \\
$t_{f}$ & time of FFT using precomputed exponentials (for time comparison only)
\end{tabular}
\vspace{1ex}
\caption{Key for column labels of Tables \ref{figrand}, \ref{figcheb}, and
\ref{figfft}. }
\label{key}
\end{table}
\begin{table}[h!]
\centering
$$
\begin{array}{r|c|c|c|c|c}
  n &  t_w &  t_p  &  t_u  &  t_d  & \epsilon_r \\  
\hline
   1000 &  0.74\E-03 &  0.18\E-02 &  0.93\E-04 &  0.66\E-03 &  0.19\E-14 \\ 
   2000 &  0.19\E-02 &  0.31\E-02 &  0.19\E-03 &  0.25\E-02 &  0.30\E-14 \\ 
   4000 &  0.42\E-02 &  0.61\E-02 &  0.43\E-03 &  0.10\E-01 &  0.52\E-14 \\
   8000 &  0.85\E-02 &  0.10\E-01 &  0.89\E-03 &  0.37\E-01 &  0.72\E-14 \\
  16000 &  0.18\E-01 &  0.25\E-01 &  0.18\E-02 &  0.14\E+00 &  0.92\E-14 \\
  32000 &  0.38\E-01 &  0.49\E-01 &  0.37\E-02 &  0.59\E+00 &  0.19\E-13 \\
  64000 &  0.84\E-01 &  0.98\E-01 &  0.78\E-02 &  0.23\E+01 &  0.21\E-13 \\
 128000 &  0.16\E+00 &  0.19\E+00 &  0.18\E-01 &  0.95\E+01 &  0.35\E-13 \\
 256000 &  0.37\E+00 &  0.53\E+00 &  0.34\E-01 &  0.40\E+02 &  0.59\E-13 \\
 512000 &  0.75\E+00 &  0.10\E+01 &  0.71\E-01 &  0.19\E+03 &  0.88\E-13 \\
1024000 &  0.17\E+01 &  0.23\E+01 &  0.15\E+00 &  0.81\E+03 &  0.14\E-12 \\
\end{array}
$$
\caption{Numerical results for uniformly random points in $[1,10]$.}
\label{figrand}
\end{table}
\begin{table}[h!]
\centering
\begin{minipage}{4.5in}
$$
\begin{array}{r|c|c|c|c|c}
     n &     t_w    &    t_p    &     t_u    &     t_d    &    \epsilon_r \\ 
\hline
   1000&  0.54\E-03 & 0.12\E-02 &  0.74\E-04 &  0.60\E-03 & 0.11\E-14 \\ 
   2000&  0.15\E-02 & 0.26\E-02 &  0.15\E-03 &  0.24\E-02 & 0.14\E-14 \\
   4000&  0.38\E-02 & 0.51\E-02 &  0.37\E-03 &  0.99\E-02 & 0.39\E-14 \\
   8000&  0.83\E-02 & 0.10\E-01 &  0.85\E-03 &  0.38\E-01 & 0.35\E-14 \\
  16000&  0.19\E-01 & 0.23\E-01 &  0.17\E-02 &  0.14\E+00 & 0.58\E-14 \\
  32000&  0.41\E-01 & 0.48\E-01 &  0.37\E-02 &  0.62\E+00 & 0.89\E-14 \\
  64000&  0.98\E-01 & 0.90\E-01 &  0.82\E-02 &  0.24\E+01 & 0.12\E-13 \\
 128000&  0.22\E+00 & 0.19\E+00 &  0.23\E-01 &  0.10\E+02 & 0.19\E-13 \\
 256000&  0.44\E+00 & 0.47\E+00 &  0.32\E-01 &  0.40\E+02 & 0.26\E-13 \\
 512000&  0.84\E+00 & 0.94\E+00 &  0.73\E-01 &  0.19\E+03 & 0.52\E-13 \\
1024000&  0.19\E+01 & 0.19\E+01 &  0.14\E+00 &  0.84\E+03 & 0.64\E-13 \\
\end{array}
$$
\end{minipage}
\vspace{1ex}
\caption{Numerical results for Chebyshev nodes on $[-1,1]$.}
\label{figcheb}
\end{table}

\begin{table}[h!]
\centering
$$
\begin{array}{r|c|c}
n & t_u & t_{f} \\
\hline
    1000 &0.91E-04 &0.16E-04 \\ 
    2000 &0.28E-03 &0.37E-04 \\
    4000 &0.41E-03 &0.44E-04 \\
    8000 &0.93E-03 &0.85E-04 \\
   16000 &0.18E-02 &0.24E-03 \\
   32000 &0.38E-02 &0.41E-03 \\
   64000 &0.81E-02 &0.88E-03 \\
  128000 &0.18E-01 &0.19E-02 \\
  256000 &0.38E-01 &0.59E-02 \\
  512000 &0.71E-01 &0.12E-01 \\
 1024000 &0.14E+00 &0.25E-01 \\
\end{array}
$$
\caption{Time comparison with FFT.} \label{figfft}
\end{table}

\subsection{Computing nodes and weights} \label{nodesandweights}
The algorithm of \S \ref{algomain} is described under the assumption that nodes
$t_1,\ldots,t_m$ and weights $w_1,\ldots,w_m$ are given such that 
\begin{equation} \label{eqapp6}
\left| \frac{1}{r} - \sum_{j=1}^m  w_j e^{-r t_j} \right| \le \varepsilon
\quad \text{for} \quad 
r \in [\delta (b-a), b-a],
\end{equation}
where $\varepsilon > 0$ and $\delta > 0$ are fixed parameters. As in the
proof of Lemma \ref{quadlem} we note that by rescaling  $r$ it suffices to find
nodes and weights satisfying 
\begin{equation} \label{eqap5}
\left| \frac{1}{r} - \sum_{j=1}^m  w_j e^{-r t_j} \right| \le \varepsilon
\quad \text{for} \quad 
r \in [1, \delta^{-1}].
\end{equation}
Indeed, if the nodes $t_1,\ldots,t_m$ and weights $w_1,\ldots,w_m$ satisfy
\eqref{eqap5}, then the nodes $t_1/(b-a),\ldots,t_m/(b-a)$ and weights
$w_1/(b-a),\ldots,w_m/(b-a)$ will satisfy \eqref{eqapp6}. Thus, in order to
implement the algorithm of \S \ref{algomain} it suffices to tabulate nodes and
weights that are valid for $r \in [1,M]$ for various values of $M$. In  the
implementation used in the numerical experiments in this paper, we tabulated
nodes and weights valid for $r \in [1,M]$ for
$$
M = [1,4^k] \quad \text{for} \quad k = 1,\ldots,10.
$$
For example, in Tables \ref{fig01} and \ref{fig02} we have listed $m = 33$ nodes
$t_1,\ldots,t_{33}$ and weights $w_1,\ldots,w_{33}$ such that 
$$
\left| \frac{1}{r} - \sum_{j=1}^{33} w_j e^{-r t_j} \right| \le 10^{-15},
$$
for all $r \in [1,1024]$.
\begin{table}[h!]
\centering
\begin{minipage}{4.5in}
{\small
\begin{verbatim}
0.2273983006898589D-03,0.1206524521003404D-02,0.3003171636661616D-02,
0.5681878572654425D-02,0.9344657316017281D-02,0.1414265501822061D-01,
0.2029260691940998D-01,0.2809891134697047D-01,0.3798133147119762D-01,
0.5050795277167632D-01,0.6643372693847560D-01,0.8674681067847460D-01,
0.1127269233505314D+00,0.1460210820252656D+00,0.1887424688689547D+00,
0.2435986924712581D+00,0.3140569015209982D+00,0.4045552087678740D+00,
0.5207726670656921D+00,0.6699737362118449D+00,0.8614482005965975D+00,
0.1107074709906516D+01,0.1422047253849542D+01,0.1825822499573290D+01,
0.2343379511131976D+01,0.3006948272874077D+01,0.3858496861353812D+01,
0.4953559345813267D+01,0.6367677940017810D+01,0.8208553424367139D+01,
0.1064261195532074D+02,0.1396688222191633D+02,0.1889449184151398D+02
\end{verbatim}
}
\end{minipage}
\vspace{1ex}
\caption{A list of $33$ nodes $t_1,\ldots,t_{33}$.} \label{fig01}
\end{table}

\begin{table}[h!]
\centering
\begin{minipage}{4.5in}
{\small
\begin{verbatim}
0.5845245927410881D-03,0.1379782337905140D-02,0.2224121503815854D-02,
0.3150105276431181D-02,0.4200370923383030D-02,0.5431379037435571D-02,
0.6918794756934398D-02,0.8763225538492927D-02,0.1109565843047196D-01,
0.1408264766413004D-01,0.1793263393523491D-01,0.2290557147478609D-01,
0.2932752351846237D-01,0.3761087060298772D-01,0.4828044150885936D-01,
0.6200636888239893D-01,0.7964527252809662D-01,0.1022921587521237D+00,
0.1313462348178323D+00,0.1685948994092301D+00,0.2163218289369589D+00,
0.2774479391081561D+00,0.3557192797195578D+00,0.4559662159666857D+00,
0.5844792718191478D+00,0.7495918095861060D+00,0.9626599456939077D+00,
0.1239869481076760D+01,0.1605927580173348D+01,0.2102583514906888D+01,
0.2811829220697454D+01,0.3937959064316012D+01,0.6294697335695096D+01
\end{verbatim}
}
\end{minipage}
\vspace{1ex}
\caption{A list of $33$ weights $w_1,\ldots,w_{33}$.} \label{fig02}
\end{table}

The nodes and weights satisfying \eqref{eqap5} can be computed by using
a procedure for generating generalized Gaussian quadratures for Chebyshev
systems together with the proof of Lemma \ref{approx}. Indeed,  Lemma
\ref{approx} is constructive with the exception of the step that invokes Lemma
\ref{krein} of Kre\u{\i}n. The procedure described in
\cite{MR2671296} is a  constructive version of Lemma \ref{krein}: given a
Chebyshev system of functions, it generates the corresponding quadrature nodes
and weights. We remark that generalized Gaussian quadrature generation codes are
a powerful tools for numerical computation with a wide range of applications.
The quadrature generation code used in this paper was an optimized version of
\cite{MR2671296} recently developed by Serkh for \cite{MR3564124}.

\subsection*{Acknowledgements}
The authors would like to thank Jeremy Hoskins for many useful
discussions.  Certain commercial equipment is identified in this paper
to foster understanding. Such identification does not imply
recommendation or endorsement by the National Institute of Standards
and Technology, nor does it imply that equipment identified is
necessarily the best available for the purpose.

\end{document}